\documentclass[11pt]{article}
\usepackage{xcolor}

\usepackage{pst-all}
\usepackage{pgfplots}
\usepackage{pst-plot}
\usetikzlibrary{calc}

\usepackage{xcolor}
\usepackage{tikz}
\usepackage{tkz-euclide}

\usepackage{bbm}
\usepackage{hyperref}
\hypersetup{
    colorlinks=true,              
    breaklinks=true,              
    urlcolor= black,              
    linkcolor= blue,              
    bookmarksopen=false,
    filecolor=black,
    citecolor=blue,
    linkbordercolor=blue}

\usepackage{chngcntr}

\usepackage{mathrsfs}

\usepackage[margin=1in]{geometry}

\usepackage{amsmath,amsthm,amssymb}

\usepackage{bbm}

\usepackage{graphicx}
\usepackage{faktor}

\newcommand{\R}{\mathbb{R}}

\newcommand{\N}{\mathbb{N}}

\newcommand{\Z}{\mathbb{Z}}

\newcommand{\SL}{\text{SL}}

\newcommand{\p}{\textbf{p}}

\newcommand{\q}{\textbf{q}}

\newcommand{\h}{\normalfont \textbf{h}}

\newcommand{\f}{\normalfont \textbf{f}}

\newcommand{\bb}{\normalfont \textbf{b}}
\newcommand{\cc}{\normalfont \textbf{c}}

\newcommand{\e}{\epsilon}

\newcommand{\de}{\delta}

\newcommand{\la}{\lambda}

\newcommand{\La}{\Lambda}

\newcommand{\Ha}{\mathscr{H}}
\newcommand{\Pa}{\mathscr{P}}

\newcommand{\D}{\mathcal{D}}

\newcommand{\Sing}{\normalfont \textbf{Sing}}
\newcommand{\VSing}{\normalfont \textbf{VSing}}
\newcommand{\BA}{\normalfont \textbf{BA}}

\newtheorem{theorem}{Theorem}[section]

\newtheorem{lemma}[theorem]{Lemma}

\newtheorem{proposition}[theorem]{Proposition}

\newtheorem{corollary}[theorem]{Corollary}

\theoremstyle{definition}

\newtheorem{definition}[theorem]{Definition}

\newtheorem{example}[theorem]{Example}

\theoremstyle{remark}

\newtheorem{remark}[theorem]{Remark}

\numberwithin{equation}{section}

\usepackage{setspace}

\usepackage[
style=alphabetic,
sorting=nty,
maxnames=99,
maxalphanames=99
]{biblatex}
\renewbibmacro{in:}{}
\DeclareFieldFormat[article]{citetitle}{\mkbibemph{#1}}
\DeclareFieldFormat[article]{title}{\mkbibemph{#1}}
\DeclareFieldFormat[article]{journaltitle}{#1}

\usepackage{enumerate}

\AtEndDocument{%
  \par
  \medskip
  \begin{tabular}{@{}l@{}}%
    \textsc{Department of mathematics, the Ohio State University}\\
    \textit{E-mail address}: \texttt{xing.211@osu.edu}
  \end{tabular}}
  
\title{Fractional dimension related to badly approximable matrices associated with higher successive minima}
\author{Hao Xing}

\begin{document}
\date{}
\maketitle

\begin{abstract}
In this article we introduce the notion of badly approximable matrices of higher order using higher sucessive minima in $\R^d$. We will prove that for order less than $d$, they have Lebesgue measure zero and the gaps between them still have full Hausdorff dimension.
\end{abstract}

\section{Introduction}

Throughout, all norms of vectors in an Euclidean space without specification are assumed to be the maximum Euclidean norm.

\subsection{Dirichlet's Theorem and Diophantine Approximation}

Diophantine approximation is a branch of number theory studying the approximation of real numbers by rational numbers. The first known result of this kind is due to Adrien-Marie Legendre, which can be proved by the Pigeonhole Principle:

\begin{theorem}[Legendre, 1808]
For any real number $x$ and every $Q > 1$, there exists an integer vector $(p, q) \in \Z^2$
such that
$$|xq-p|<\frac{1}{Q} \text{ and } 0<q<Q.$$
\end{theorem}

The Classical multidimensional version of approximation theorem, Due to Johann Peter Gustav Lejeune Dirichlet, states that

\begin{theorem}[Dirichlet's approximation theorem \cite{Di1842}, 1842]
For every real $m \times n$ matrix $A$ and every
$Q > 1$, there exists an integer vector
$ \p\in Z^m$ and $\q\in Z^n$ such that
$$\| A\q-\p \|<\frac{1}{Q^{\frac{n}{m}}} \text{ and } 0<\q<Q.$$
\end{theorem}
The modern proof of Dirichlet's approximation theorem often uses a generalized version of ``Pigeonhole Principle" for lattices in Euclidean spaces, called the (first) convex body theorem, due to Hermann Minkowski, applied to an appropriately chosen convex set\footnote{For the special case when $m=n=1$, this convex set can be chosen as
$$S = \left\{ (x,y) \in \R^2 : -N-\frac{1}{2} \leq x \leq N+\frac{1}{2}, \vert \alpha x - y \vert \leq \frac{1}{N} \right\}.$$}:

\begin{theorem}[Minkowski \cite{MI1896}, 1889]
Suppose that $L$ is a unimodular lattice in $\R^d$ of covolume $1$ and $S$ is a convex subset of $\R^d$ that is symmetric with respect to the origin (namely $x \in S$ if and only if $-x\in S$). If the volume of $S$ is strictly greater than $2^n$, then $S$ must contain at least one lattice point other than the origin. 
\end{theorem}

\subsection{Matrix diophantine approximation and Dani's correspondence principle}
We begin by introducing two classes of matrices that generalizes the Diophantine approximation of real numbers.

\begin{definition}
Let $M$ denote the set of all $m \times n$ matrices with real entries. A matrix $A \in M$ is called \textit{singular} if for all $\epsilon > 0$, there exists $Q_\epsilon$ such that for all $Q \ge Q_{\epsilon}$, there exist integer vectors $\p \in \mathbb Z^m$ and $\q \in Z^n$ such that
\begin{equation} \label{eq:1.1}
    \|A\q+\p\|\le \e Q^{-n/m} ~\text{and}~
0<\| \q \|< Q
\end{equation}

Here $\|\cdot\|$ denotes an arbitrary norm on $\R^m$ or $\R^n$. We denote the set of singular $m \times n$ matrices by
$\textbf{Sing}(m, n)$.

An $m \times n$ matrix $A$ is called \textit{badly approximable} if there exists $c > 0$ such that for all integer vectors $\p \in \Z^m$ and $\q \in \Z^n-\{0\}$ we have 
$$ \|A \q + \p \| \ge c \| \q \|^{-\frac{n}{m}}. $$

We denote the collection of badly approximable $m \times n$ matrices by $\textbf{BA}(m,n)$.
\end{definition}

\begin{example}
When $m=n=1$, being singular is the same as being rational. If $A$ is a rational number, then it trivially satisfies \eqref{eq:1.1}. Conversely, by Hurwitz's approximation theorem \cite{BE02}, given any irrational number $\alpha$, there exist infinitely many rational numbers $p/q$ with $(p,q)=1$ such that 
$$|\alpha-\frac{p}{q}|<\frac{1}{\sqrt{5}q^2},$$

and that $\sqrt{5}$ is optimal. So if $\e <\frac{1}{\sqrt{5}}$, then there exist integers $p,q$ with $(p,q)=1$ such that $|\alpha-\frac{p}{q}|>\frac{\e}{\sqrt{5}}$ which fails \eqref{eq:1.1}.

By Liouville’s theorem \cite{EW11}, any quadratic irrational (solutions to quadratic equations over $\mathbb Z$) is badly approximable, but it is unknown whether all algebraic numbers are badly approximable.
\end{example}

\vspace{5mm}
The sets of badly approximable and singular matrices are linked to homogeneous dynamics via the Dani
correspondence principle.\label{gt} For each $t \in \R$ and for each matrix $A$, let $g_t:=\begin{bmatrix}
e^{t/m}I_m & 0 \\
0 & e^{-t/n}I_n 
\end{bmatrix}$ 
and 
$u_A:=\begin{bmatrix}
I_m & A \\
0 & I_n 
\end{bmatrix}$, where $I_k$ denotes the $k$-dimensional identity matrix.

By the Dani's correspondence principle (\cite{Dani1985DivergentTO}), the Diophantine properties of $A$ and the dynamical properties of the orbit $(g_t u_A x_0)_{t\ge 0}$ (which consists of unimodular lattices) can be summarized in the following table. Write $x_0=\Z^{m+n}$ ,an element in the space of unimodular lattices in $\R^{m+n}$, also identified with the neutral element in $\SL(m+n,\R)/\SL(m+n,\Z)$:

\begin{table}
\begin{center}
\setlength\tabcolsep{20pt} 

\begin{tabular}{|c|c|}
\hline
Diophantine properties of $A$ & Dynamical properties of $(g_t u_A x_0)_{t\ge 0}$       \\ \hline
$A$ is badly approximable     & $(g_t u_A x_0)_{t\ge 0}$ is bounded                    \\ \hline
$A$ is singular               & $(g_t u_A x_0)_{t\ge 0}$ is divergent                  \\ \hline
\end{tabular}
\caption{\label{example-table}Dani's correspondence.}
\end{center}
\end{table}

\vspace{3mm}

\begin{theorem}
Let $G$ be a Lie group and $\Gamma$ be a unimodular Lattice in $G$. Consider the $G$-homgeneous space $X:=G/\Gamma$, equipped with the $G$-invariant Borel probability measure $\mu_X$. Then the $G$-action on $X$ is ergodic.
\end{theorem}

\begin{proof}
Let $\mu_G$ be a left Haar measure on $G$. Let $f\in L^2(X,\mu_X)$ be such that $f$ is invariant under $G$-action, namely for any $g\in G$
$$f(g.x)=f(x), \text{for}~\mu_X \text{-almost every}~ x.$$

We will show that $f$ is a constant almost everywhere with respect to $\mu_X$.

Consider the product space $G\times X$. We will apply Fubini-Tonelli Theorem to the function $|f(g.x)-f(x)|$ on $G\times X$:
$$0=\int_G \int_X |f(g.x)-f(x)| d\mu_X d\mu_G = \int_X \int_G |f(g.x)-f(x)| d\mu_G d\mu_X.$$

Hence, $0=\int_G |f(g.x)-f(x)| d\mu_G(g)$ for $\mu_X$-almost every $x\in X$. Therefore for almost all $x\in X$ (and in particular there exists $x\in X$), there exists $U\subset G$ with $\mu_G(U)=1$ such that for every $g\in U$,
$$f(g.x) = f(x).$$

\vspace{3mm}
That $f$ is $\mu_X$-almost everywhere constant follows from the claim below.
\vspace{3mm}

Claim: $\mu_X(U.x)=1$ for any $x\in X$. 

To show this claim, we use the quotient integral formula (See for example Theorem 1.5.3 in \cite{DE14} or Theorem 2.51 in \cite{Fo15}):

Take $h=1_{(G-U)g_0}$, noticing that $\mu((G-U)g_0)=0$ because of the unimodularity. Now let $\mu_\Gamma$ be a Haar measure on $\Gamma$, by the the quotient integral formula and Fubini's theorem we have:
$$0=\int_G h d\mu_G = \int_{G/\Gamma} \int_{\Gamma} h(g\gamma) d\mu_\Gamma ~d\mu_X =  \int_{\Gamma}\int_{G/\Gamma} h(g\gamma)d\mu_X~ d\mu_\Gamma.$$

Since $\Gamma$ is countable, the Haar measure on it must be a scalar multiple of counting measure. So we must have for every $\gamma \in \Gamma$, in particular for $\gamma = e$,
$$0=\int_{G/\Gamma} h(g)d\mu_X= \mu_X((G-U)g_0 \Gamma).$$

Since $g_0\in G$ is arbitrary, and by the transitivity of the action, $G/\Gamma - Ux$ is contained in $(G-U).x$, we are done.

\end{proof}

\begin{theorem}[See Chapter III Corollary 2.2 of \cite{BM00}]\label{howe}
If a simple Lie group with finite center acts ergodically on a probability space $X$, then every subgroup of $G$ with a non-compact closure is strongly mixing, and thus ergodic on $X$.
\end{theorem}

It follows from the ergodicity of $(g_t)$-action on $\SL(m+n,\R)/\SL(m+n,\Z)$ and the equidistribution of orbits under ergodic actions that $\BA(m,n), \Sing(m,n)$ and $\VSing(m,n)$ all have Lebesgue measure zero.

\vspace{5mm}
To further investigate the sets with various Diophantine properties, we will look into their fractional dimensions. In order to compute the Hausdorff dimension of badly approximable numbers and matrices, Schmidt invented the following topological game \cite{Schmidt65onbadly}  \cite{Sch69} called \text{Schmidt Games}:

Choose two parameters $0<\alpha<1$ and $0<\beta<1$. Two players, called Alice and Bob, will play the following game:

\begin{itemize}
  \item First Bob choose a closed ball $B_1$ in $\R^d$;
  \item Then Alice choose closed ball $A_1 \subset B_1$ in $\R^d$ whose radius is $\alpha$ times the radius of $B_1$
  \item Next Bob chooses a closed ball $B_2 \subset A_1$ whose radius is $\beta$ times the radius of $A_1$
  \item Then Alice chooses a closed ball $A_2 \subset B_2$ whose radius is $\alpha$ times the radius of $B_2$
  \item $\cdots \cdots$
\end{itemize}
We call a sequence of choices by Alice (resp. Bob) depending on the choices of Bob (resp. Alice) a \textit{strategy}. A set $S\subset \R^d$ is called \textit{(Alice)-winning} if Alice has a strategy to make sure (no matter how Bob chooses his strategy), we have
$$\cap_{k=1}^{\infty}A_k \subset S.$$

Schmidt proved the following 

\begin{theorem}[Theorem 2, \cite{Sch69}]
The set of badly approximable matrices in $\R^{m\times n}$ is a winning set.
\end{theorem}

\begin{theorem}[Corollary 2 to Theorem 6, \cite{Schmidt65onbadly}]
Any winning set in an Euclidean space is of full Hausdorff dimension.
\end{theorem}

By introducing a modified version of Schmidt's game, Kleinbock and Weiss \cite{KLEINBOCK20101276} proved that the set of weighted badly approximable matrices is also winning and thus of full Hausdorff dimension. Specifically,

\begin{theorem}
For $r_i,s_j, 1\le i\le m, 1\le j \le n$ with $\sum_{i=1}^m r_i=1=\sum_{j=1}^n s_j$, let $\textbf{r}=(r_1,\dots,r_m)$ and $\textbf{s}=(s_1,\dots,s_n)$. Then the set of badly approximable matrices with weight $(\textbf{r},\textbf{s})$, denoted $$\BA^{\textbf{r},\textbf{s}}(m,n):=\{A\in \R^{m\times n}: \inf_{\p\in \Z^m, \q \in \Z^n-\{0\}} \|Aq-p\|\textbf{r}\cdot \|q\|_{\textbf{s}}>0 \},$$ 
where the notation $\|x\|_{\textbf{r}}:=\max\{|x_1|^{\frac{1}{r_1}},\cdots,|x_m|^{\frac{1}{r_m}}\}$ for any $x\in \R^m$,
is a winning set of modified Schmidt game and thus of full Hausdorff dimension $mn$.
\end{theorem}

However, it is a major challenge in Diophantine approximations to compute the Hausdorff dimension of the set $\normalfont{\textbf{Sing}}(m, n)$ of singular matrices. The first breakthrough was made in 2011 by Cheung \cite{Cheung2007HausdorffDO}, who proved that the Hausdorff dimension of $\textbf{Sing}(2, 1)$ is $4/3$; this was extended in 2016 by Cheung and Chevallier \cite{Cheung2016HausdorffDO}, who proved that the Hausdorff dimension of $\textbf{Sing}(m, 1)$ is $m^2/(m+1)$ for
all $m \ge 2$; while Kadyrov, Kleinbock, Lindenstrauss, and Margulis \cite{KKLM}  proved that the Hausdorff dimension of $\textbf{Sing}(m, n)$ is at most $\de_{m,n}:=mn(1-\frac{1}{m+n})$. Most recently, Das, Fishman, Simmons and Urbański \cite{DFSU20} proved that this upper bounded is sharp. Their proof is based on a generalized variational principle and is independent of the previous results.

\begin{theorem} [\cite{DFSU20}]
For all $(m, n) \ne (1, 1)$, we have
$$\dim_H(\normalfont{\textbf{Sing}}(m, n)) = \dim_P(\normalfont{\textbf{Sing}}(m, n)) = \delta_{m,n},$$

where $\dim_H(S)$ and $\dim_P(S)$ denote the Hausdorff and packing dimensions of a set $S$, respectively.
\end{theorem}

\begin{remark}
When $m=n=1, \dim_H(\normalfont{\textbf{Sing}}(m, n)) = \dim_P (\normalfont{\textbf{Sing}}(m, n))=0$, since in this case $\normalfont{\textbf{Sing}}(m, n)$ is simply the set of rational numbers.
\end{remark}

\vspace{1cm}
\subsection{Successive minima functions and matrices with Diophantine Approximation properties of higher orders}

\begin{definition} \label{sm}
Let $d = m + n$, and for each $j = 1,\dots, d,$
let $\lambda_j (\Lambda)$ denote the $j$-th minimum of a lattice $\Lambda \subset \mathbb R^d$
(with respect to the $l^2$ norm on $\mathbb R^d$ \footnote{Note that $\|\cdot\|_{\infty} \le \|\cdot\|_{2} \le \sqrt{d} \|\cdot\|_{\infty} $. So if we use the maximum norm $\|\cdot\|:=\|\cdot\|_{\infty}$ to define $\la_i$, the resulting $\la_i^{\infty}$ is equivalent to $\la_i$ up to a multiple constant depending on $d$, which doesn't change any results below. We use $l^2$ norm here since it is most common in literature.}), i.e. the
infimum of $\lambda$ such that the set $\{r \in \Lambda : \|r\|_2 \le  \lambda \}$ contains $j$ linearly independent vectors.
	
For a $m\times n$ matrix $A$, the \textit{successive minima function} of the matrix A, denoted 
$\h=\h_A=(h_1,\dots,h_d):[0,\infty) \to \mathbb R^d$ is defined by the formula
\begin{equation}
h_i(t):=\log \lambda_i(g_t u_A \Z^d).
\end{equation}

\end{definition}

Then the Dani's correspondence principle can be translated into the language of successive minima function as follows:

\begin{table}
\begin{center}
\setlength\tabcolsep{20pt} 
\begin{tabular}{|c|c|}

\hline
Diophantine properties of $A$ & Dynamical properties of $(g_t u_A x_0)_{t\ge 0}$       \\ \hline
$A$ is badly approximable     & $\sup_{t\ge 0} -\h_{A,1(t)}<\infty$            \\ \hline
$A$ is singular               & $\lim_{t\to \infty} -\h_{A,1(t)}=\infty$                   \\  \hline
\end{tabular}
\caption{Dani's correspondence with successive minima function.}
\end{center}
\end{table}

In light of successive minima functions, we can generalize the notion of badly approximable matrices, singular matrices as follows:

\begin{definition}
For $r=1,2,\dots,d=m+n$, A matrix $A \in M(m\times n ,\mathbb R)$ is called \textbf{badly approximable of order $r$} if 
$$sup_{t\ge 0} -\h_{A,r}(t)<\infty.$$

$A \in M(m\times n ,\mathbb R)$ is called \textbf{singular of order $r$} if 
$$\lim_{t\to \infty} -\log \h_{A,r}(t)=\infty .$$

Let $\BA_r(m,n)$ (resp. $\Sing_r(m,n)$) denote the set of badly approximable (resp. singular) $m\times n$ matrices. $\BA_r(m,n)$ (resp. $\Sing_r(m,n)$) form an ascending (descending) sequence of sets in $r$.
\end{definition}

\subsection{Dani's correspondence for Diophantine
Approximation properties of higher orders}

For $r=1,2,...,d$ and a lattice $\La$, let $I^r(\La)$ denote the set of all $r$-tuples of linearly independent vectors $(v_1,\dots,v_r)$. \label{the notation I^r}

\begin{theorem}
A matrix $A\in \R^{m\times n}$ is badly approximable of order $r$ if and only if there exists $c>0$ such that for all linearly independent $r$ vectors $(\p_1,\q_1),\dots, (\p_r,\q_r)\in Z^m\times (\Z^n-\{0\})$, there exists $1\le i \le r$ satisfying 
\begin{equation*}
    \|A\q_i-\p_i\|\ge \frac{c}{\|\q_i\|^{\frac{n}{m}}}.
\end{equation*}
\end{theorem}

\begin{proof}
For the forward implication, notice that the definition of badly approximable of order $r$ is equivalent to saying that there exists $\de>0$ such that 
\begin{equation*}
    \la_r(g_t u_A \Z^d)\ge \de, \forall t>0.
\end{equation*}
Using our $I^r$ notation, this is the same as saying there exists $\de>0$ such that 
\begin{equation}\label{eq:intersection of I^r and B is empty}
    I^r(g_tu_A\Z^d)\cap B_{\de}^r= \varnothing, \forall t\ge 0.
\end{equation}
where $B_{\de}$ denotes the (open) ball in $\R^d$ centered at the origin with radius $\de$ and $B_{\de}$ means its $r$-fold Cartesian product.

But 
\begin{align}\label{eq:computatin of product}
    g_tu_A
    \begin{bmatrix}\p \\
    \q 
    \end{bmatrix}  
&=  \begin{bmatrix}
e^{\frac{t}{m}}I_m &  \\
& e^{-\frac{t}{n}}I_n
\end{bmatrix}   
\begin{bmatrix}
I_m & A  \\
 & I_n
\end{bmatrix} 
\begin{bmatrix}\p \\
\q 
\end{bmatrix}=
\begin{bmatrix}e^{\frac{t}{m}}(\p+A\q) \\
e^{-\frac{t}{n}}\q 
\end{bmatrix}.
\end{align}

So the above equation \ref{eq:intersection of I^r and B is empty} implies (from here we change the $2$-norm to $\infty$-norm) that there exists $\de>0$ such that for all linearly independent $r$ vectors $(\p_1,\q_1),\dots, (\p_r,\q_r)\in Z^m\times (\Z^n-\{0\})=\Z^d$, there exists $1\le i \le r$ satisfying
\begin{align*}
    & \|e^{\frac{t}{m}}(A\q_i+\p_i)\|\ge \de \text{ or }\\
    & \|e^{-\frac{t}{n}}\q_i\| \ge \de
\end{align*}
for all $t\ge 0$. Note that since $\q_i \ne 0$, we can choose $t=t(\de,q_i)$ so that 
$\|e^{-\frac{t}{n}}\q_i\|=\frac{\de}{2}<\de$, then the second possibility is blocked and 
$$\|A\q_i+\p_i\|\ge e^{-\frac{t}{m}}\de=(e^{-\frac{t}{n}})^{\frac{n}{m}} \de = \left(\frac{\de/2}{\|\q_i\|}\right)^{\frac{n}{m}}\de =:\frac{c}{\|\q_i\|^{\frac{n}{m}}}.$$
For the backward implication, there exists $c>0$ such that for all linearly independent $r$ vectors $(\p_1,\q_1),\dots, (\p_r,\q_r)\in Z^m\times (\Z^n-\{0\})$, there exists $1\le i \le r$ satisfying 
\begin{equation*}
    \|A\q_i-\p_i\|\ge \frac{c}{\|\q_i\|^{\frac{n}{m}}}.
\end{equation*}

We want to find $\de>0$ such that 
$$\la_r(g_tu_A\Z^d)\ge \de,$$
for all $t\ge 0$. From the computation of product \ref{eq:computatin of product}, this is the same as there exists $\de>0$, such that for all $t\ge 0$,
\begin{equation*}
   \la_r\left( \begin{bmatrix}e^{\frac{t}{m}}(\p+A\q) \\
e^{-\frac{t}{n}}\q 
\end{bmatrix} : \p \in \Z^m, \q \in \Z^n \right)\ge \de
\end{equation*}
Suppose on the contrary that this is not possible, then for any $\de>0$, there exists $t\ge 0$ and linearly independent $r$ vectors $(\p_1,\q_1),\dots, (\p_r,\q_r)\in Z^m\times \Z^n$ satisfying
\begin{align*}
    & \|e^{\frac{t}{m}}(A\q_i+\p_i)\|< \de \text{ and }\\
    & \|e^{-\frac{t}{n}}\q_i\| < \de
\end{align*}
for all $1\le i \le r.$
But the first equation times the $\frac{n}{m}$-th power of the second equation yields 
\begin{equation*}
    \|A\q_i+\p_i\|\cdot \|\q_i\|^{\frac{n}{m}}<\de^{\frac{m+n}{m}}
\end{equation*}
for $i=1,2,...,r$, contradicting to the statement after ``if and only if" whenever $c<\de^{\frac{m+n}{m}}$.
\end{proof}

\begin{theorem}
A matrix $A\in \R^{m\times n}$ is singular of order $r$ if and only if for any $\e>0$ there exists $Q_{\e}>0$ such that for all $Q>Q_{\e}$, there exist $r$ linearly independent vectors $(\p_1,\q_1),\dots, (\p_r,\q_r)\in \Z^m\times (\Z^n-\{0\})$ satisfying 
\begin{align*}
    &\|A\q_i-\p_i\|\le \frac{\de}{\|\q_i\|^{\frac{n}{m}}},\\ 
    \text{and }& 0<\|\q_i\|<Q.
\end{align*}
for all $1 \le i \le d$.
\end{theorem}
\begin{proof} We first prove the forward direction. The aim is to find $Q_{\e}$.
Fix $\e>0$, the definition of singularity of order $r$ is the same as saying that for any $\de>0$, there exists $T_{\de}>0$ such that for all $t\ge T_{\de}$,
\begin{equation*}
    \la_r(g_t u_A \Z^d)< \de.
\end{equation*}
In view of the equation \ref{eq:computatin of product}, this is equivalent to that for any $\de>0$, there exists $T_{\de}>0$ such that for all $t\ge T_{\de}$, there exist linearly independent $r$ vectors $(\p_1,\q_1),\dots, (\p_r,\q_r)\in Z^m\times \Z^n=\Z^d$ with 
\begin{equation}\label{the less than inequality for singular}
    \begin{cases}
     \|e^{\frac{t}{m}}(A\q_i+\p_i)\|< \de \\
     \|e^{-\frac{t}{n}}\q_i\| < \de
\end{cases}
\end{equation}
for all $1\le i\le d$.

In order to find the $Q_{\e}$ we need, we consider the system of inequality
\begin{equation}\label{solve for Q}
    \begin{cases}
        e^{-\frac{t}{m}}\de \le \e Q^{-\frac{n}{m}}\\
        e^{\frac{t}{n}}\de \le Q
    \end{cases}
\end{equation}
and solve it for $Q$. Note that this is equivalent to 
\begin{equation}\label{inequality for Q}
  \de e^{\frac{t}{n}}  \le Q \le \left( \frac{\e}{\de} \right)^{\frac{m}{n}}e^{\frac{t}{n}}
\end{equation}

We first fix $\de:= \frac{1}{2}\e^{\frac{m}{m+n}} < \e^{\frac{m}{m+n}}$ so that \ref{inequality for Q} is solvable for $Q$ and it follows that as long as  
$$Q\in I=\cup_{t\ge T_{\de}} I(t),$$
where
$$I(t):=\left[\de e^{\frac{t}{n}}, \left( \frac{\e}{\de} \right)^{\frac{m}{n}}e^{\frac{t}{n}} \right],$$
then \ref{inequality for Q} holds. Therefore, our choice of $Q_{\e}$ in the statement of the theorem can be 
$$Q_{\e}:=\frac{1}{2}\e^{\frac{m}{m+n}}e^{\frac{T_{\de}}{n}}.$$

For the backward direction, suppose now for any $\e>0$ there exists $Q_{\e}>0$ such that for all $Q>Q_{\e}$, there exist $r$ linearly independent vectors $(\p_1,\q_1),\dots, (\p_r,\q_r)\in \Z^m\times (\Z^n-\{0\})$ satisfying 
\begin{align}\label{eq: conditions for singular}
    &\|A\q_i-\p_i\|\le \frac{\de}{\|\q_i\|^{\frac{n}{m}}} \text{ and } \nonumber\\ 
    & 0<\|\q_i\|<Q,
\end{align}
for all $1 \le i \le d$.

For any $\de>0$, we want to find $T_{\de}>0$ such that for any $t\ge T_{\de}$, 
$$\la_r(g_tu_A\Z^d)<\de.$$

Again from \ref{eq:computatin of product}, what we need is 
\begin{equation*}
    \begin{cases}
     \|e^{\frac{t}{m}}(A\q_i+\p_i)\|< \de \\
     \|e^{-\frac{t}{n}}\q_i\| < \de
\end{cases}
\end{equation*}
From \ref{eq: conditions for singular}, we solve
\begin{equation}\label{solve for t}
    \begin{cases}
         e^{\frac{t}{m}} \e Q^{-\frac{n}{m}} < \de\\
         e^{-\frac{t}{n}} Q < \de
    \end{cases}
\end{equation}
for $t$ as
$$n\log \left(\frac{Q}{\de}\right)<t<m\log \left(\frac{\de}{\e} Q^{\frac{n}{m}} \right )$$
This is solvable as long as we choose $\e:=\frac{1}{2}\de^{\frac{m+n}{m}}<\de^{\frac{m+n}{m}}$. Let $$J=\cup_{Q\ge Q_{\e}}J_Q,$$
where $J_Q:=\left[n\log \left(\frac{Q}{\de}\right), m\log \left(\frac{\de}{\e} Q^{\frac{n}{m}} \right )\right]$. It follows that the $T_{\de}$ we need can be taken as 
$$T_Q:=n\log \left(\frac{Q_{\e}}{\de}\right).$$
\end{proof}

\vspace{1cm}
\subsection{Hausdorff and packing dimensions and the statement of main results \label{1.2}} 

The \textit{s-dimensional Hausdorff measure} of a set $S \subset \R^D$ is defined to be
\begin{equation}
    \Ha^s(S):=\lim_{\delta \to 0}\inf\left \{\sum_{i=1}^\infty (\operatorname{diam} U_i)^s: \bigcup_{i=1}^\infty U_i\supseteq S, \operatorname{diam} U_i<\delta\right \}
\end{equation}

The \textit{s-dimensional packing measure} of a set $S$ is defined as
\begin{equation}
  \Pa^s (S) = \inf \left\{ \left. \sum_{j \in J} \Pa_0^s (S_j) \right| S \subseteq \bigcup_{j \in J} S_j, J \text{ countable} \right\},
\end{equation}
where $\Pa^S_0$, called the \textit{s-dimensional packing pre-measure}, is defined as
\begin{equation}
    \Pa_0^s (S) = \limsup_{\delta \to 0}\left\{ \left. \sum_{i \in I} \mathrm{diam} (B_i)^s \right| \begin{matrix} \{ B_i \}_{i \in I} \text{ is a countable collection} \\ \text{of pairwise disjoint closed balls with} \\ \text{diameters } \leq \delta \text{ and centres in } S \end{matrix} \right\}.
\end{equation}

Given the measures defined above, we define the Hausdorff dimension and packing dimension of a set $S\subset \R^D$ as follows:
\begin{equation}
    \dim_{\mathrm{H}}(S)=\inf\{d\ge 0:\Ha^d(S)=0\}=\sup\left(\{d\ge 0:\Ha^d(S)=\infty\}\cup\{0\}\right),
\end{equation}
\begin{equation}
\dim_{\mathrm{P}} (S)  = \sup \{ s \geq 0 | \Pa^s (S) = + \infty \}  = \inf \{ s \geq 0 | \Pa^s (S) = 0 \}.
\end{equation}

Our first result is on the measure of badly approximable matrices of higher orders:

\begin{theorem} \label{bar}
	$\BA_{d}(m,n)=\R^{m\times n}$. For all $r=1,2,\dots, d-1$, $\BA_r(m,n)$ is Lebesgue null. 
\end{theorem} 

Since $\BA_1(m,n) \subset \BA_2(m,n) \cdots \subset \BA_d(m,n)=\R^{m\times n}$, a natural question to ask next is how big the Gaps $BA_{r+1}(m,n)-BA_{r}(m,n)$ are in terms of fractional dimensions. We have

\begin{theorem}\label{hbar}
For $1\le r \le d-1$, we have the Hausdorff dimension for the gaps between the badly approximable matrices of order $r$ and $r+1$ is full:
$$\dim_H\big(\BA_{r+1}(m,n)-\BA_{r}(m,n) \big)=\dim_P\big(\BA_{r+1}(m,n)-\BA_{r}(m,n) \big)=mn.$$
\end{theorem}

\begin{remark}
    As of now, we are not able to obtain an explicit formula for the Hausdorff and packing dimensions for the singular matrices of order $r$. This is a work in progress.
\end{remark}

\newcommand\locally{\mathrel{\overset{\makebox[0pt]{\mbox{\normalfont\tiny\sffamily locally}}}{=}}}

\section{Variational principles and the proof of the main results}

\subsection{Templates}

Recall the definition of successive minina (Definition \ref{sm}). The key idea of variational principles is to approximate successive minina function $\h$ by piecewise linear functions called templates, define appropriate averaging quantities for templates and study the relation between such quantities and the fractional dimension.

\begin{definition}[\cite{DFSU20} \label{templates}]
An $m \times n$ \textit{template} is a piecewise linear map $\f : [0,\infty) \to \R^d$ with the
following properties:
 \renewcommand{\labelenumi}{(\Roman{enumi})}
\begin{enumerate}
    \item $f_1 \le \cdots \le f_d$.
    \item $-\frac{1}{n} \le f_i' \le \frac{1}{m}$ for all $i$.
    \item For all $j = 0, \dots , d$ and for every interval $I$ such that $f_{j} < f_{j+1}$ on $I$, the function

$F_j:=\sum_{0<i \le j} f_i$ is convex and piecewise linear on $I$ with slopes in
\begin{equation}
    Z(j):=\{\frac{L_+}{m}- \frac{L_-}{n}:L_{\pm}\in [0,d_{\pm}]_{\Z}, L_+ + L_- =j \}
\end{equation}

\end{enumerate}

  Here $d_+:=m, d_-:=n, [a,b]_{\Z}:=[a,b]\cap \Z$.
  
  We use the convention that $f_0 = -\infty$ and $f_{d+1} = +\infty$. We will call the assertion that $F_j$ is convex
the \textit{convexity condition}, and the assertion that its slopes are in $Z(j)$ the \textit{quantized
slope condition}. We denote the space of $m\times n$ templates by $\mathcal{T}_{m, n}$

Observe that $F_d=f_1 +... +f_d$ is always a constant due to the property (III) above. A template $\f$ is called \textit{balanced} if $F_d=0$. A \textit{partial template} is a piecewise linear map $f$ satisfying (I)-(III) whose
domain is a closed, possibly infinite, subinterval of $[0,\infty)$.
\end{definition}

The fundamental relation between templates and successive minima functions is given
as follows:

\begin{theorem}[\cite{DFSU20}] \label{2.3} \leavevmode
 \renewcommand{\labelenumi}{(\roman{enumi})}
\begin{enumerate}
    \item For every $m \times n$ matrix $A$, there exists an $m \times n$ template $\f$ such that $\h_A \asymp_+  \f$.
    \item For every $m \times n$ matrix $A$, there exists an $m \times n$ template $\f$ such that $\h_A \asymp_+ \f$.
\end{enumerate}
 
\end{theorem}

Theorem \ref{2.3}(ii) asserts that for every template f, the set 
$\mathcal{D}(f) := \{A : \h_A \asymp_+ f \}$
is nonempty. It is natural to ask how big this set is in terms of Hausdorff and packing
dimensions. Moreover, given a collection of templates $\mathcal{F}$, we can ask the same question about the set 
$$\mathcal{D}(\mathcal{F}):=\cup_{\f\in \mathcal{F}} D(f).$$

\subsection{The lower and upper contraction rates}
The next important notion we need to introduce for the statement of the variational principal is the lower and upper average contraction rate of a template.

We define the lower and upper average contraction rate of a template $\f$ as follows. Let $I$ be an open interval on which $\f$ is linear. For each $q = 1, \dots , d$ such that $f_q < f_{q+1}$ on $I$, let $L_{\pm}=L_{\pm}(\f,I,q)\in [0,d_{\pm}]_{\Z}$ be chosen to satisfy $L_+ + L_- =q$ and 
\begin{equation}
F_q'=\sum_{i=1^q}f_i'=\frac{L_+}{m}-\frac{L_-}{m},
\end{equation}

as guaranteed by (III) of the definition of templates. An interval of equality for $f$ on $I$ is an interval $(p, q]_{\Z}$, where $0 \le p < q \le d$ satisfy 
\begin{equation}
f_p<f_{p+1}=\cdots=f_q<f_{q+1}~\text{on}~I,
\end{equation}

As before, we use the convention that $f_0=-\infty$ and $f_{d+1}=\infty$. Note that the collection of intervals of equality forms a partition of $[1,d]_{\Z}$. If $(p,q]_{\Z}$ is an interval of equality for $\f$ on $I$, then we let $M_{\pm}(p,q)=M_{\pm}(\f, I, p,q)$, where 
\begin{equation}
M_{\pm}(\f, I, p, q)=L_{\pm}(\f,I,q)-L_{\pm}(\f,I,p),
\end{equation}
or equivalently, $M_{\pm}(p, q)$ are the unique integers such that
\begin{equation}
M_{+}+M_-=q-p ~\text{and}~ 
\sum_{i=p+1}^{q}f_i'=\frac{M_+}{m}-\frac{M_-}{n} ~\text{on}~I.
\end{equation}
It can be shown from the definition of template that $M_{\pm}\ge 0$ by (II) of the definition of templates. Next, let
\begin{equation}
S_+=S_+(\f, I)=\cup_{(p,q]_{\Z}}(p,p+M_+(p,q)]_{\Z}
\end{equation} 
\begin{equation}
S_-=S_-(\f, I)=\cup_{(p,q]_{\Z}}(p+M_+(p,q),q]_{\Z}
\end{equation} 

where the unions are taken over all intervals of equality for $\f$ on $I$. Note that $S_+$ and $S_-$
are disjoint and satisfy 
$S_+ \cup S_- = [1, d]_{\Z}$, and that $\#(S+) = m$ and $\#(S_-) = n$. 

Next, let
\begin{equation}
\delta(\f,I)=\# \{(i_+,i_-)\in S_+\times S_-:i_+< i_-\}\in [0,mn]_{\Z}, \label{counting}
\end{equation}

and note that 
\begin{equation}
mn-\delta(\f,I)=\#\{(i_+,i_-)\in S_+\times S_-:i_+ > i_-\}
\end{equation}

\begin{definition}\label{contraction rates}
The \textit{lower and upper average contraction rates} of $\f$ are the numbers 

\begin{equation}
\underline{\delta}(\f):=\liminf_{T \to \infty} \Delta(\f,T),
\end{equation}
and 
\begin{equation}
\overline{\delta}(\f):=\limsup_{T \to \infty} \Delta(\f,T),
\end{equation} 
where $\Delta(\f,T):=\frac{1}{T}\int_0^T \delta(\f,t)dt$. Here we abuse notation by writing $\delta(\f, t) = \delta(\f, I)$ for all $t \in I$.
\end{definition}

To help illustrate the definitions above, let us introduce the following example

\begin{example}[The contraction rate of zero template is $mn$ \label{zerotpl}]
    Let $\f$ be the template where all of its components are zero, then the interval of linearity is $[0,\R)$ and the only interval of equality is $(0,d]_{\Z}$.
    
\begin{enumerate}
    \item For $q=0$
    \begin{equation*}
    \begin{cases}
    L_+ +L_-=0\\
    \frac{L_+}{m}-\frac{L_-}{n}=0
    \end{cases}\iff 
    \begin{cases}
    L_+=0\\
    L_-=0
    \end{cases}.
\end{equation*}

\item For $q=d$
    \begin{equation*}
    \begin{cases}
    L_+ +L_-=F_d'=0\\
    \frac{L_+}{m}-\frac{L_-}{n}=0
    \end{cases}\iff 
    \begin{cases}
    L_+=m\\
    L_-=n
    \end{cases}.
\end{equation*}
Hence,
$$M_+(0,d)=M_{+}(\f,[0,\R),0,d):=L_+(d)-L_+(0)=m$$
and
$$S_+=\cup_{(p,q]_{\Z}}(p,p+M(p,q)]=(0,m]_{\Z}$$
and
$$S_-=[1,d]_{\Z}-S_+=(m+1,d].$$
It follows that
\begin{equation*}
\delta(\f,I)=\# \{(i_+,i_-)\in S_+\times S_-:i_+< i_-\}=mn.
\label{counting}
\end{equation*}
\end{enumerate}
Hence
$$\de(\f)=\lim_{T\to \infty}\frac{1}{T}\int_0^T \delta(\f,t)dt\equiv mn. $$
\end{example}

\begin{example}[The contraction rate of standard quadrilateral partial template of order $r$ \label{quadtpl}]

    For $r\le \min(m,n)$, let $\f$ be the template with
    $$f_{1}=\cdots=f_r<f_{r+1}=\cdots=f_d,$$    
    where each of the components is a piecewise linear function with two pieces defined on $I_1:=[0,\frac{m}{m+n}]$ and $I_2:=[\frac{m}{m+n},1]$ (intervals of linearity). For $f_{1}=\cdots=f_r$ the derivative (slope) on the first interval is $-\frac{1}{n}$ and the derivative on the second interval is $\frac{1}{m}$; for $f_{r+1}=\cdots=f_d$ the derivative (slope) on the first interval is $-\frac{r}{(d-r)n}$ and the derivative on the second interval is $-\frac{r}{(d-r)m}$.
    
    There are two intervals of equality: $(0,r]_{\Z}$ and $(r+1,d]_{\Z}$.

\begin{figure}[h]
\begin{center}
\psset{unit=1.5cm}
    \begin{pspicture}[showgrid=false](0,-0.75)(4,2)
    \psaxes[yAxis=false, ticks=none, labels=none, linecolor=black]{->}(0,0)(-1,-0.5)(6,1.5)[$t$,0][,-10]
    \psset{algebraic,linewidth=1.5pt,linecolor=blue}
    \pscustom
    {
        \psplot{0}{3}{-x/3}
        \psplot{3}{5}{(x-5)/2}{$f_1$}   
        \rput(1,1){$y=\log x$}
    }
    \psset{algebraic,linewidth=1.5pt,linecolor=red}
    {
        \psplot{0}{3}{x/6}
        \psplot{3}{5}{-(x-5)/4}
    }
        \pcline[linestyle=none](2,3)(0,0)\naput{$f_{r+1}=\cdots=f_d$}
        \pcline[linestyle=none](2,1.7)(0,0)\naput{$\frac{r}{(d-r)n}$}
        \pcline[linestyle=none](2,-2)(2,0.25)\naput{-$\frac{1}{n}$}
        \pcline[linestyle=none](6,0)(2,-1)\naput{$\frac{1}{m}$}
        \pcline[linestyle=none](6,1)(2,0.75)\naput{$-\frac{r}{(d-r)m}$}
        \pcline[linestyle=none](5,-2.5)(0,0)\naput{$f_{1}=\cdots=f_r$}

\end{pspicture}
\end{center}
    \vspace{10mm}
    \caption{Construction of the standard partial template of order $r$.}
    \label{fig:pqtpl}
\end{figure}
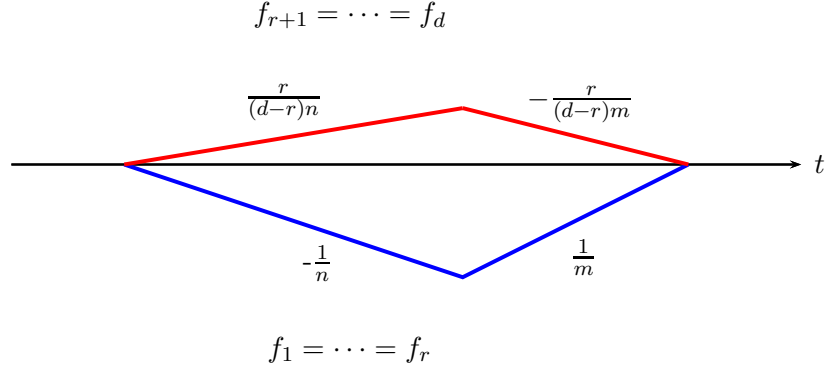

\end{example}

\begin{enumerate}
    \item Over the interval $I_1$:
\begin{enumerate}
    \item For $q=0$
    \begin{equation*}
    \begin{cases}
    L_+ +L_-=0\\
    \frac{L_+}{m}-\frac{L_-}{n}=0
    \end{cases}\iff 
    \begin{cases}
    L_+=0\\
    L_-=0
    \end{cases}.
\end{equation*}

\item For $q=r$
    \begin{equation*}
    \begin{cases}
    L_+ +L_-=r\\
    \frac{L_+}{m}-\frac{L_-}{n}=F_r'=\sum_{i=1}^r f_i'=-\frac{r}{n}
    \end{cases}\iff 
    \begin{cases}
    L_+=0\\
    L_-=r
    \end{cases}.
    \end{equation*}
\item For $q=d$
    \begin{equation*}
    \begin{cases}
    L_+ +L_-=d\\
    \frac{L_+}{m}-\frac{L_-}{n}=0
    \end{cases}\iff 
    \begin{cases}
    L_+=m\\
    L_-=n
    \end{cases}.
    \end{equation*}
Hence,
\begin{align*}
    &M_+(0,r)=M_{+}(\f,[0,\R),0,r):=L_+(r)-L_+(0)=0;\\
    &M_+(r,d)=M_{+}(\f,[0,\R),r,d):=L_+(d)-L_+(r)=m.
\end{align*}
and
$$S_+=\cup_{(p,q]_{\Z}}(p,p+M(p,q)]=(r,m+r]_{\Z}$$
and
$$S_-=[1,d]_{\Z}-S_+=(0,r]_{\Z}\cup(m+r+1,d]_{\Z}.$$
It follows that
\begin{equation*}
\delta(\f,I_1)=\# \{(i_+,i_-)\in S_+\times S_-:i_+< i_-\}=m(n-r).
\label{counting}
\end{equation*}
\end{enumerate}

    \item Over the interval $I_2$:
\begin{enumerate}
    \item For $q=0$
    \begin{equation*}
    \begin{cases}
    L_+ +L_-=0\\
    \frac{L_+}{m}-\frac{L_-}{n}=0
    \end{cases}\iff 
    \begin{cases}
    L_+=0\\
    L_-=0
    \end{cases}.
\end{equation*}

\item For $q=r$
    \begin{equation*}
    \begin{cases}
    L_+ +L_-=r\\
    \frac{L_+}{m}-\frac{L_-}{n}=F_r'=\sum_{i=1}^r f_i'=+\frac{r}{m}
    \end{cases}\iff 
    \begin{cases}
    L_+=r\\
    L_-=0
    \end{cases}.
    \end{equation*}
\item For $q=d$
    \begin{equation*}
    \begin{cases}
    L_+ +L_-=F_d'=d\\
    \frac{L_+}{m}-\frac{L_-}{n}=0
    \end{cases}\iff 
    \begin{cases}
    L_+=m\\
    L_-=n
    \end{cases}.
    \end{equation*}
Hence,
\begin{align*}
    &M_+(0,r)=M_{+}(\f,0,r):=L_+(r)-L_+(0)=r;\\
    &M_+(r,d)=M_{+}(\f,r,d):=L_+(d)-L_+(r)=m-r.
\end{align*}
and
\begin{align*}
&S_+=\cup_{(p,q]_{\Z}}(p,p+M(p,q)]=(0,r]_{\Z}\cap(r,m]_{\Z}=(0,m]_{\Z}\\
&S_-=[1,d]_{\Z}-S_+=(m+1,d]_{\Z}.
\end{align*}

It follows that
\begin{equation*}
\delta(\f,I_2)=\# \{(i_+,i_-)\in S_+\times S_-:i_+< i_-\}=mn.
\label{counting}
\end{equation*}
\end{enumerate}
\end{enumerate}
Therefore,
$$\int_0^1 \delta(\f,t) dt= \frac{n}{m+n}m(n-r)+\frac{m}{m+n}mn=mn-\frac{mnr}{m+n}.$$

Note that when $r=1$, this is equal to $mn-\frac{mn}{m+n}$. This example plays a central role in the computation of the Hausdorff dimension of singular matrices.

Note that if $r> \min(m,n)$, then the derivatives (slopes) will violate the axiom $-\frac{1}{n}\le f_i' \le \frac{1}{m}, i=1,2,\cdots,d$ for templates.

\subsection{The variational principles for templates}

\begin{definition}
A collection of templates $\mathcal{F}$ is said to be \textit{closed under finite perturbations} if whenever $g \asymp_+ \f \in \mathcal{F}$, we have $g \in F$.
\end{definition}

\begin{theorem}[Variational principle: version 1 \cite{DFSU20}]
Let $\mathcal F$ be a collection of templates
closed under finite perturbations. Then
\begin{equation}
    \dim_H(\mathcal{D}(\mathcal F))=\sup_{\f \in \mathcal F} \underline{\delta}(\f)
\end{equation}

and
\begin{equation}
           \dim_P(\mathcal{D}(\mathcal F))=\sup_{\f \in \mathcal F} \overline{\delta}(\f)
\end{equation}
\end{theorem}

\begin{theorem}[Variational principle, version 2 \cite{DFSU20}]\label{variational principle version 2}. Let S be a collection of Borel functions from $[0,\infty)$ to $\R^d$ which is closed under finite perturbations, and let $\D(S):=\{A:\h_A \in S\}$. Then
\begin{equation}
    \dim_H(\mathcal{D}(S))=\sup_{\textbf{f} \in S\cap \mathcal{T}_{m,n}} \underline{\delta}(\f)
\end{equation}

and 
\begin{equation}
        \dim_P(\mathcal{D}(S))=\sup_{\textbf{f} \in S\cap \mathcal{T}_{m,n}} \overline{\delta}(\f)
\end{equation}

\end{theorem}

\begin{theorem}[Variational principle, version 3 \cite{DFSU20}] 

\leavevmode
 \renewcommand{\labelenumi}{(\roman{enumi})}
\begin{enumerate}
    \item Let S be a (Borel) set of $m \times n$ matrices of Hausdorff (resp. packing) dimension $>\de$.
Then there exist a matrix $A \in S$ and a template $\f \asymp_+ \h_A$ whose lower (resp. upper)
average contraction rate is $> \de$.

    \item Let $\f$  be a template whose lower (resp. upper) average contraction rate is $> \de$. Then
there exists a (Borel) set S of m×n matrices of Hausdorff (resp. packing) dimension $>\de$, such that $\h_A \asymp_+ \f$ for all $A \in S$.
\end{enumerate}

\end{theorem}

\subsection{Proof of Main Theorems}

\subsection{Proof of Theorem \ref{bar}}

\begin{lemma}\label{lma:pu}
For $d=m+n$, $G=\SL(d,\R)$, let $P_-$ denote the subgroup of lower diagonal block matrices
\[
  P_- :=\left\lbrace \begin{bmatrix}
A & O \\
C & D \end{bmatrix}\in \SL(d,\mathbb Z) \;\middle|\;
  \begin{tabular}{@{}l@{}}
    $A,C,D$ are any real matrices of size $m\times m,n\times m, n\times n$ respectively,\\
    and $O$ is a zero $m\times n$ matrix
   \end{tabular}
  \right\rbrace
\]
and $U^+$ denote the subgroup of unipotent upper diagonal matrices
\[
  U^+ :=\left\lbrace \begin{bmatrix}
I_m & B \\
O & I_m \end{bmatrix}\in \SL(d,\mathbb Z) \;\middle|\;
  \begin{tabular}{@{}l@{}}
 $B$ is a any $m\times n$ real matrix.
   \end{tabular},
  \right\rbrace
\]
then the complement $G-P_-U^+$ has Haar measure zero in $G$.
\end{lemma}

\begin{proof}
The classical proof of this involves the theory of algebraic geometry and algebraic groups. By looking at the orbit $P_- \backslash  P_-U^+$ in the irreducible projective variety $P_- \backslash G$. By the theory of algebraic groups, $P_- \backslash  P_-U^+$ is Zariski open in its closure and further Zariski open, and therefore dense in $P_- \backslash G$. Therefore its complement in $P_- \backslash G$ has strictly non-full dimension. So $P_-U^+$ has non-full dimension in $G$ and thus of zero Haar measure. 

For the self-containedness purpose,we shall give an elementary proof here that only involves linear algebra and manifold theory. We will study the set of matrices $\begin{bmatrix}
X & Y \\
Z & W \end{bmatrix}$ in $\SL(n,\Z)$ that can be (and cannot be) represented by matrices in $P_-$ and $U^+$. 

Suppose $\begin{bmatrix}
A & O \\
C & D \end{bmatrix}\in P_-$ and $\begin{bmatrix}
I_m & B \\
O & I_n \end{bmatrix} \in U^-$ and we have the following equation of block matrices:

\begin{equation}\label{eq:matrix}
    \begin{bmatrix}
X & Y \\
Z & W \end{bmatrix}=\begin{bmatrix}
A & O \\
C & D \end{bmatrix} \begin{bmatrix}
I_m & B \\
O & I_n \end{bmatrix}=\begin{bmatrix}
A & AB \\
C & BC+D \end{bmatrix}
\end{equation}

\vspace{5mm}
We have the system of equations:

\begin{equation}
    \begin{cases}
        X=A\\
        Z=C\\
        Y=AB\\
        W=BC+D
    \end{cases}\,
\end{equation}
which can be simplified to $Y=XB$ and $W=BZ+D$. Therefore, the matrix equation \ref{eq:matrix} is solvable for $A,B,C,D$ and  if and only if $Y=XB$ is solvable for $B$ (noticing that $W=BZ+D$ always gives the solution $D=W-BZ$). However, if $Y=XB$ is not solvable for $B$, then we must have $\det(X)=0$ for the $m\times m$ matrix $X$ from the beginning. It follows that 
\begin{align*}
    &G-P_-U^+ \\ =& \left \{\begin{bmatrix}
X & Y \\
Z & W \end{bmatrix}\in \SL(d,\R): \begin{bmatrix}
X & Y \\
Z & W \end{bmatrix}~\text{cannot be written as}~\begin{bmatrix}
A & O \\
C & D \end{bmatrix} \begin{bmatrix}
I_m & B \\
O & I_n \end{bmatrix}=\begin{bmatrix}
A & AB \\
C & BC+D \end{bmatrix}  \right \}\\
\subset & \left \{\begin{bmatrix}
X & Y \\
Z & W \end{bmatrix}\in \SL(d,\R): \det(X)=0  \right \}
\end{align*}
By the last term is a subvariety of $\SL(d,\R)$ whose dimension is strictly less than the full dimension, and therefore of Haar measure zero since the Haar measure on Lie groups are given by the full-dimensional volume form.

\end{proof}

\begin{lemma}[Decomposition of Haar measure in Lie groups, Theorem 8.32 in \cite{KN02}]
Let $G$ be a Lie group, and let $S$ and $T$ be closed subgroups such that $S \cap T$ is compact, multiplication $S \times T \to G$ is 
an open map, and the set of products ST exhausts $G$ except possibly for a  set of Haar measure 0. Let $\Delta_T$ and $\Delta_G$ denote the modular functions of $T$ 
and $G$. Then the left Haar measures on $G$, $S$, and $T$ can be normalized so 
that 
$$\int_G f(x) dx = \int_{S\times T}f(st)\frac{\Delta_T(t)}{\Delta_G(t)} dsdt,$$
for all Borel function $f\ge 0$ on $G$.
\end{lemma}

\begin{remark}
    We can apply this lemma to the scenario where $S=P_-$ and $T=U^+$ since they are both closed subgroups with trivial (compact) intersection. That $P_-\times U^+ \to G$ is open follows from the fact that this map is injective (since $S=P_-$ and $T=U^+$ have trivial intersection) and thus an immersion for Lie groups. Also the product maps in Lie groups are submersions. Therefore it gives an local diffeomorphsm and thus open. Finally, the previous lemma gives $G-P_-U^+$ is of measure zero. 
\end{remark}

\begin{theorem}[Birkhoff's Pointwise Ergodic Theorem, the discrete version,\cite{EW11},Theorem 2.30] Let $(X,\mathscr{B},\mu)$ ba a probability space and $T:X\to X$ be an ergodic measure-preserving transformation. If $f\in L^1(X)$, then 
$$\lim_{n\to \infty} \frac{1}{n} \sum_{i=0}^{n-1}f(T^ix)=\int_X f d\mu,$$
for $\mu$-almost every $x\in X$.
\end{theorem}

\begin{corollary}[Birkhoff's Pointwise Ergodic Theorem, the continuous version \label{corcts}] Let $G$ be a topological group with continuous action on $X$,namely $G\times X\to X, (g,x)\to g.x$ is continuous and let $(X,\mathscr{B},\mu)$ a probability space with $\mu$ a $G$-invariant probability measure . Suppose that $(g_t)_{t\in \R}$ is a one-parameter subgroup of $G$ and that there the discrete subgroup $(g_{n})_{n\in Z}$ acting ergodically on $X$.

If $f\in L^1(X)$, then 
$$\lim_{t\to +\infty} \frac{1}{T} \int_0^{T} f(g_t.x)=\int_X f d\mu,$$
for $\mu$-almost every $x\in X$.
\end{corollary}

\begin{proof}
First, observe that from the definition of ergodicity we have immediately that $(g_{n})_{n\in Z}$ acting ergodically on $X$  implies that $g_1$ acts on $X$ ergodically. Now taking $T:=g_1$ and $f_s(x):=\int_0^s f(g_t.x)dt$ in the discrete version of Birkorff's ergodic theorem, noticing that $f_s$ is again a $L^1$-function since $s$ is fixed and 
$$\int_X f_s(x) d\mu=\int_X \int_0^s f(g_t.x)dt d\mu=\int_0^s \int_X  f(g_t.x)d\mu dt<\infty,$$
by Fubini's theorem.

The discrete version of Birkorff's ergodic theorem gives us
$$\lim_{n\to \infty} \frac{1}{n}\sum_{i=0}^{n-1}f_s(g_{is}.x)=\int_X f_s d\mu$$
Therefore for this fixed $s$, we have
\begin{align*}
    \lim_{T\to \infty}\frac{1}{T}\int_0^T f(g_t.x)dt 
    =& \lim_{n\to \infty} \frac{1}{s\lfloor \frac{T}{s} \rfloor} \left (\sum_{i=0}^{\lfloor \frac{T}{s} \rfloor-1}\int_{is}^{(i+1)s} f(g_{t}.x)dt + O_x(1) \right)\\
    =& \lim_{n\to \infty} \frac{1}{s\lfloor \frac{T}{s} \rfloor} \left (\sum_{i=0}^{\lfloor \frac{T}{s} \rfloor-1}\int_0^s f(g_{t+is}.x)dt + O_x(1) \right)\\
    =& \lim_{n\to \infty} \frac{1}{s\lfloor \frac{T}{s} \rfloor} \left (\sum_{i=0}^{\lfloor \frac{T}{s} \rfloor-1} f_s(g_{is}.x)dt + O_x(1) \right)\\
    =&\frac{1}{s}\int_X f_s d\mu \tag{by the discrete Birkorff}\\
    =&\frac{1}{s}\int_X \int_0^s f(g_t.x)dt d\mu \\
    =&\frac{1}{s}\int_0^s \int_X  f(g_t.x)d\mu dt =\frac{1}{s}\int_0^s \int_X  f(g_t.x)d\mu dt \tag{by Fubini}\\
    =& \frac{1}{s}\int_0^s \int_X  f(x)d\mu dt
    \tag{$\mu$ is $g_t$-invariant}\\
    =& \int_X f(x)d\mu.
\end{align*}

\end{proof}

\begin{lemma}\label{positive}
For any $r=1,2,\dots d-1$ and $\delta>0$, the set
$B_{\delta}^r:= \{\Lambda \in \mathcal{L}:\lambda_r(\Lambda) < \delta\}$ has positive measure, where the measure is the unique $\SL(d,\R)$ invariant measure in the homogeneous space
$G/\Gamma:=\SL(d,\R) / \SL(d,\Z)$.

\end{lemma}

\begin{proof}
First we observe that for $r=1,2,\dots d-1$, $B_{\delta}^r$ is nonempty since it contains the elements
$$\frac{\de}{2}e_1,\cdots,\frac{\de}{2}e_{d-1},$$
which already form a linearly independent set in $B_\delta^r$ of lengths all less than $\de$. Note that for $\de$ small enough, this cannot be generalized to $r=1,2,\dots, d$ due the Minkowski's second convex body theorem \ref{Min2}.

By the continuity of $\lambda_r$ on $G/\Gamma$ (by the Theorem \ref{thm:cts}), $B_\de^r$ is open. 

Any open subset on $G/\Gamma$. Indeed, since any open subset of ${\SL(d,\R)}/{\SL(d,\Z)}$ has countably many translations under $\text{SL}(n,\mathbb Q)$ whose union will cover the whole space, whose measure is $1$, and this follows from the $G$-invariance of the measure. 
\end{proof} 

\begin{proof}[Proof of the Theorem \ref{bar}]
We first notice that the condition
$$\limsup_{t\to \infty} -\h_{A,r}(t)<\infty.$$
is equivalent to $$\inf_{t\ge 0} \lambda_r(g_t u_A \Z^d) \ge \delta,$$
for some $\delta>0$.

When $r=d(=m+n)$, we notice that by the Theorem \ref{Min2} (Minkowski's second convex body theorem), $$[\lambda_d(g_t u_A \Z^d)]^d \ge \prod_{r=1}^d\lambda_r(g_t u_A \Z^d)\asymp_d \text{covol}(g_t u_A \Z^d)=1.$$

Hence there exist a $\delta_d > 0$ such that for any $m\times n$ matrix $A$ and $t\ge 0$, 
$$\inf_{t\ge 0} \lambda_d(g_t u_A \Z^d) \ge \delta_d.$$
Namely $\BA_{d}(m,n)=\R^{m\times n}$.

For $r\le d-1$, we first observe that the set 
$\{\Lambda \in \mathcal L: \inf_{t\ge 0} \lambda_r(g_t \Lambda)\ge \delta\}$ is $g_t$-invariant for any $t\ge 0$. So by the ergodicity of $(g_n)_{n\in \Z}$-action (Theorem \ref{howe}), this set has $\mu$-measure zero or 1. So it suffices to show its complement $\{\Lambda \in \mathcal L: \inf_{t\ge 0} \lambda_r(g_t \Lambda)< \delta\}$ has positive measure for any $\delta>0$. However, by the continuous version of Birkoff's ergodic theorem, Corollary \ref{corcts},
$$\lim_{T\to \infty}\frac{1}{T}\int_0^T f(g_t.\Lambda)dt=
\int_X f(\Lambda)d\Lambda,$$
for $\mu$-almost every $\Lambda \in \mathcal{L}$ (identified with $X:= \SL(d,\R) /\SL(d,\Z)$ and $f\in L^1(\mathcal{L},\mu)$.

Now take $f$ as the characteristic function on the set $B_{\delta}^r:=\{\Lambda \in \mathcal{L}:\lambda_r(\Lambda) < \delta\}$, which is of positive measure by Lemma \ref{positive} and it follows that for $\mu$-a.e. $\Lambda \in \mathcal{L}$,
\begin{equation}
    \lim_{T\to \infty}\frac{1}{T}\int_0^T \textbf{1}_{B_{\delta}^r}(g_t.\Lambda)dt >0.
\end{equation}
This means for almost every lattice $\Lambda \in \mathcal{L}$ and over a positive proportion of time $t \in [0,T]$ when $T$ is large, $\lambda_r(g_t\Lambda) < \delta$. More precisely, thanks to the boundedness of $\textbf{1}$, for $\mu$-a.e. $\Lambda$ and any $T_0>0$, 
$$\lim_{T\to \infty}\frac{1}{T} \int_{T_0}^T \textbf{1}_{B_{\delta}^r}(g_t.\Lambda) dt >0.$$

In particular, for $\mu$-a.e. $\Lambda$ and any $T_0>0$, there exists $t>t_0$ such that 
\begin{equation}\label{limsup}
    \limsup_{t\to \infty} f(g_t.\Lambda)<\delta.
\end{equation}
This together with the fact that the set $B_{\delta}^r:=\{\Lambda \in \mathcal{L}:\lambda_r(\Lambda) < \delta\}$, which is of positive measure gives  
$$\mu \left(\{\Lambda \in \mathcal L: \inf_{t\ge 0} \lambda_r(g_t \Lambda)< \delta\} \right)\ge
\mu \left(\{\Lambda \in \mathcal L: \limsup_{t\ge 0} \lambda_r(g_t \Lambda)< \delta\} \right )>0$$

Therefore by the ergodicity of $g_t$-action, $\{\Lambda \in \mathcal L: \inf_{t\ge 0} \lambda_r(g_t \Lambda)\ge \delta\}$ has zero $\mu$-measure in the space of unimodular lattices.

It remains for us to show that the set of matrices corresponding to the lattices $\{u_A \Z^d\}$ has measure zero.

To this end, we first recall the root space decomposition for $\frak{sl}(d,\R)$, cf. \cite{KN02} Chapter II section 1:
\begin{align*}
    \frak{sl}(d,\R) 
    &= \frak{h}\oplus_{i\ne j} \frak{g}_{ij} \\
    &= (\oplus_{i\ne j:1\le i \le m, 1\le j\le n;~m+1\le i \le m+n, m+1\le j \le m+n}\frak{g}_{ij}) \oplus\frak{h}  \oplus (\oplus_{ 1\le i \le m,m+1\le j \le m+n} \frak{g}_{ij}) \\
    &=  \frak{p}_- \oplus \frak{u}_+.
\end{align*}

Under the exponential map $\exp:\frak{sl}(d,\R) \to \SL(d,\Z)$, $\frak{p}_-$ corresponds to the subgroup $P_-$ in $\SL(d,\R)$, and $\frak{u}_+$ corresponds to the subgroup $U^+$ in $\SL(d,\R)$, cf. lemma \ref{lma:pu}. In particular, 
$$g_t\in \exp{\frak{p}_-},u_A\in \exp{\frak{u}_+} $$
Since the canonical quotient map $\pi:\SL(d,\R) \to \SL(d,\R)/\SL(d,\Z)$ is a local diffeomorphism of manifolds, for any $A\in M(m\times n,\R)$, we have that the map
\begin{align*}
\pi \circ \exp:\frak{sl}(d,\R)=\frak{p}_- \oplus \frak{u}_+ &\longrightarrow  \SL(d,\R)/\SL(d,\Z)\\
(X_-,X_+) &\longmapsto \exp{X_-} \exp{X_+} \cdot u_A\SL(d,\Z)
\end{align*}
gives a local coordinate at the point $u_A\SL(d,\Z)$.  

Observing that 
$$g_t\exp{X_-} \exp{X_+} \cdot u_A\SL(d,\Z)=g_t\exp{X_-} g_{-t}\cdot g_t\exp{X_+} \cdot u_A\SL(d,\Z)$$
and that
$$\left\{g_t\exp{X_-}g_{-t}\right\}_{t\ge 0}$$
is bounded, since $\exp{X_-}$ is a lower triangular block matrix contracted to the identity matrix as $t \to \infty$ under the $\{g_t\}_{t\ge 0}$ conjugation, and in view of the inequality in the Lemma \ref{bound}, we have for fixed $X_-,A$
$$\lambda_r\left(g_t u_A\SL(d,\Z)\right)~\text{is bounded below away from zero} $$
if and only if
$$\lambda_r \left(g_t\exp{X_-} \cdot u_A\SL(d,\Z)\right)~\text{is bounded below away from zero}.$$
By what we proved above using the ergodicity of $(g_t)$-action
$$\mu \{x\in G/\Gamma: \inf_{t\ge 0}\lambda_r \left(g_t x\right) > 0\}=0$$

Now notice that the subgroup $U^+:=\left \{ \begin{bmatrix}
I_m & A \\
0 & I_n 
\end{bmatrix}:A\in M(m\times n,\R)  \right\}$ is naturally isomorphic to $M(m\times n,\R)$ with multiplication corresponding to the addition. The Haar measure on $U^+$, by the uniqueness, can be identified with the Lebesgue measure on $M(m\times n,\R)$. Notice that with the local diffeomorphism $\pi:\SL(d,\R) \to \SL(d,\R)/\SL(d,\Z)$, the identification also can also go from the Haar measure on $U^+\Gamma/\Gamma$ to the Lebesgue measure on $M(m\times n,\R)$.

Now by Lemma \ref{lma:pu}, almost every element $g\Gamma$ in $G/\Gamma:=\SL(d,\R) / \SL(d,\Z)$ has the decomposition
$g\Gamma=pu\Gamma$ with $p\in P_-$ and $u\in U^+$. By the decomposition of Haar measure $dx=\frac{\Delta_{U^+}(u)}{\Delta_{G}(u)}dpdu$ in $G$ and the local identification of Haar measures on $G$ and $G/\Gamma$, we have that
\begin{align*}
    0=&\int_G \mathbbm{1}_{\{x\in G/\Gamma: \inf_{t\ge 0}\lambda_r \left(g_t x \right) > 0\}}dx \tag{by the ergodicity of $g_t$-action}\\
    = & \int_{U^+} \int_{P_-} \mathbbm{1}_{\{pu\Gamma\in G/\Gamma: \inf_{t\ge 0}\lambda_r \left(g_t pu\Gamma \right) > 0\}}\frac{\Delta_{U^+}(u)}{\Delta_{G}(u)} dp du \tag{by the decomposition of Haar measure}\\
    = & \int_{P_-}\int_{U^+}   \mathbbm{1}_{\{pu\Gamma\in G/\Gamma: \inf_{t\ge 0}\lambda_r \left(g_t pu\Gamma \right) > 0\}}\frac{\Delta_{U^+}(u)}{\Delta_{G}(u)} du dp \tag{by the Fubini's theorem}\\
    = & \int_{P_-}\int_{U^+}  \mathbbm{1}_{\{pu\Gamma\in G/\Gamma: \inf_{t\ge 0}\lambda_r \left(g_t u\Gamma \right) > 0\}}\frac{\Delta_{U^+}(u)}{\Delta_{G}(u)}du dp \tag{by the contraction above on $g_t$-contraction}\\
    = & \int_{P_-}\int_{U^+}  \mathbbm{1}_{\{u\Gamma\in G/\Gamma: \inf_{t\ge 0}\lambda_r \left(g_t u\Gamma \right) > 0\}}\frac{\Delta_{U^+}(u)}{\Delta_{G}(u)}du dp \tag{since $pu\Gamma \in G/\Gamma$ if and only if $u\Gamma \in G/\Gamma$}\\
\end{align*}
Therefore,
    $$\int_{U^+}  \mathbbm{1}_{\{u\Gamma\in G/\Gamma: \inf_{t\ge 0}\lambda_r \left(g_t u\Gamma \right) > 0\}}\frac{\Delta_{U^+}(u)}{\Delta_{G}(u)} du=0,  $$
    
and thus $\mathbbm{1}_{\{u\Gamma\in G/\Gamma: \inf_{t\ge 0}\lambda_r \left(g_t u\Gamma \right) >0\}}=0$ for $\mu$-almost all $u\in U^+$ or equivalently 
    $$\mu_{U^+}(\{pu\Gamma\in G/\Gamma: \inf_{t\ge 0}\lambda_r \left(g_t u\Gamma \right)>0\})= 0.$$

Finally, by the identification between the Haar measure on $U^+$ and the Lebesgue measure on $\R^{m\times n}$, we have
    $$ m(\{A \in\text{some neighborhood of O in }~M(m\times n,\R): \inf_{t\ge 0}\lambda_r \left(g_t u_A\Gamma \right) > 0\})=0 $$

Note that the above argument also works if we replace $\Gamma$ with $u_{A_0}\Gamma$ for some $A_0 \in M(m\times n,\R)$ and a countable union of zero-measure sets is again of zero measure. This completes the proof of the Theorem \ref{bar}.


\end{proof}

\subsection{Proof of Theorem \ref{hbar}}

To illustrate the idea of the proof, we first observe that the Examples \ref{zerotpl} and \ref{quadtpl} allows us to construct an ``electrocardiography" template with deeper and deeper ``pulse" (corresponding to quadrilateral partial templates) that spends larger and larger proportion of time with zero templates. However, the limitation for this construction, as discussed in the Example \ref{quadtpl}, is that we will violate the bounds for the derivatives of components for such templates when $r>\min(m,n)$. So the main task left for us is to lower the slopes for the templates so that they are confined in $[-\frac{1}{n},\frac{1}{m}]$ and at the same time make sure we still have the ``quantized" accumulated slope as required in the definition of templates, cf. Definition \ref{templates}.   

To this end, we need to following lemma from the elementary number theory, which will later allow us to generalize the bound on $r$ to $r\le \max(m,n)$:

\begin{lemma}\label{lma:elem}
For positive integers $m,n,r$ with $r\le \max(m,n)$, we have
\begin{equation}
    r-n \le \left \lceil \frac{mr}{m+n}-1 \right \rceil \
\end{equation}
and dually
\begin{equation}
    r-m \le \left \lceil \frac{nr}{m+n}-1 \right \rceil \
\end{equation}
\end{lemma}
\begin{proof}For the first inequality, if $r\le n$, then it trivially holds (notice that the right hand side is always nonnegative). Hence it suffices to assume $n< r \le m$.
\begin{align*}
& r-n\le \left \lceil \frac{mr}{m+n}-1 \right \rceil \\
 \iff & r-n\le \left \lceil r-\frac{nr}{m+n}-1 \right \rceil\\
 \iff & r-n\le \left \lceil -\frac{nr}{m+n}-1 \right \rceil+r \tag{since $\left \lceil x+r \right  \rceil =\left \lceil x \right  \rceil+r$}\\
 \iff & -n\le \left \lceil -\frac{nr}{m+n}-1 \right \rceil \\
  \iff & 0\le \left \lceil n-\frac{nr}{m+n}-1 \right \rceil \\
  \impliedby & 0\le \left \lceil n-\frac{nm}{m+n}-1 \right \rceil  \tag{by the monotonicity of ceiling function}\\
\end{align*}
The nonnegativeness of the right hand side on the last line follows from $\frac{nm}{m+m}<n$. This proves the first inequality and the second follows from switching $m$ and $n$.

\end{proof}

\begin{proof}

Recall the variational principle gives us
\begin{equation*}
    \dim_H(\mathcal{D}(S))=\sup_{\textbf{f} \in S\cap \mathcal{T}_{m,n}} \underline{\delta}(\f)
\end{equation*}
and 
\begin{equation*}
        \dim_P(\mathcal{D}(S))=\sup_{\textbf{f} \in S\cap \mathcal{T}_{m,n}} \overline{\delta}(\f)
\end{equation*}
where $S$ is a (Borel) collection of functions closed under finite perturbation and $\D(S):=\{A\in \R^{mn}:\h_A \in S\}$, where $\h_A$ is the successive minima function of $A$. In this proof we shall take
$$S:=\{g\in C([0,\infty),\R^d): \liminf_{x\to \infty} g_i(x) =-\infty, \forall i\le r; \liminf_{x\to \infty} g_{i}(x)>-\infty,\forall i\ge r+1\}.$$
Then it follows that 
$$\mathcal{D}(S)=\BA_{r+1}(m,n)-\BA_r(m,n)$$

We first observe that for any $m\times n$ template in $\f$, we have by definition of lower and upper contraction rates:
$$\de(\f)\le mn.$$

In view of the variational principles,it suffices to find a template (or a sequence of templates) whose contraction rates is equal to (or approximates) $mn$. 

Let $\tau_1,\tau_2$ be two positive number, which will represent slopes for templates, to be determined later. We now construct a piecewise linear map $\f$ as follows:

For convenience let us denote 
$$a_k=k^k-k, b_k:=k^k,$$

and $b_0=0$. Observe that $b_{k-1}\le a_k$ for all $k\ge 1$.

For $n=1,2,3,4\cdots$, we first define the restriction of $\f$ on $[b_k-k,b_k]$, denoted $\q_k:=\q_k[\tau_1,\tau_2]$ as follows

For the first $r$ components of $\q_k$,
\begin{equation}
q_{k1}=q_{k2}\cdots=q_{kr}=
    \begin{cases}
        -\frac{\tau_1}{r}\left(x-b_k+n\right) & \text{if }   x\in [b_k-k,\eta_k(\tau_1,\tau_2)]\\
        \frac{\tau_2}{d-r}\left (x-b_k \right) & \text{if }  x\in [\eta_k(\tau_1,\tau_2),b_k]
    \end{cases},
\end{equation}
where $\eta_k[\tau_1,\tau_2]$ is the point satisfies 
$$\frac{-b_k+k+\eta_k(\tau_1,\tau_2)}{-\eta_k(\tau_1,\tau_2)+b_k}=\frac{\tau_2}{\tau_1}.$$

For the last $d-r$ components of $\q_k$, we set
$$q_{n,r+1}=\cdots=q_{nd}$$
so that $$q_{k1}+q_{k2}+\cdots+q_{kd}\equiv 0.$$

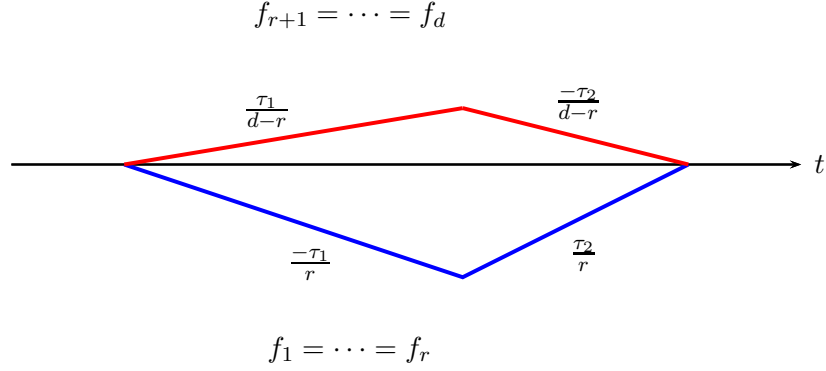
\begin{figure}
\begin{center}
\psset{unit=1.5cm}
    \begin{pspicture}[showgrid=false](0,-0.75)(4,2)
    \psaxes[yAxis=false, ticks=none, labels=none, linecolor=black]{->}(0,0)(-1,-0.5)(6,1.5)[$t$,0][,-10]
    \psset{algebraic,linewidth=1.5pt,linecolor=blue}
    \pscustom
    {
        \psplot{0}{3}{-x/3}
        \psplot{3}{5}{(x-5)/2}{$f_1$}   
        \rput(1,1){$y=\log x$}
    }
    \psset{algebraic,linewidth=1.5pt,linecolor=red}
    {
        \psplot{0}{3}{x/6}
        \psplot{3}{5}{-(x-5)/4}
    }
        \pcline[linestyle=none](2,3)(0,0)\naput{$f_{r+1}=\cdots=f_d$}
        \pcline[linestyle=none](2,1.5)(0,0)\naput{$\frac{\tau_1}{d-r}$}
        \pcline[linestyle=none](2,-2)(2,0.25)\naput{$\frac{-\tau_1}{r}$}
        \pcline[linestyle=none](6,0)(2,-1)\naput{$\frac{\tau_2}{r}$}
        \pcline[linestyle=none](6,1)(2,0.75)\naput{$\frac{-\tau_2}{d-r}$}
        \pcline[linestyle=none](5,-2.5)(0,0)\naput{$f_{1}=\cdots=f_r$}

\end{pspicture}
\end{center}
    \vspace{10mm}
    \caption{Construction of $\q[\tau_1,\tau_2]$.}
    \label{fig:qtpl}
\end{figure}

and then we define
\begin{equation}
\f=
    \begin{cases}
        \textbf{0} & \text{if }   x\in [b_{k-1},b_k-k]\\
         \textbf{q}_k[\tau_1,\tau_2](x) & \text{if }  x\in [b_k-k,b_k]
    \end{cases}.
\end{equation}

It is easy to see that 
$$\lim_{n\to \infty} f_r(\eta_k(\tau_1,\tau_2))=-\infty,$$
and therefore
$$\liminf_{t\to +\infty} f_r(t)=-\infty.$$

Now let us see under what circumstances will $\f$ indeed becomes a template on $[0,\infty)$. Recall that we have three conditions for a template:

 \renewcommand{\labelenumi}{(\Roman{enumi})}
\begin{enumerate}
    \item $f_1 \le \cdots \le f_d$.
    \item $-\frac{1}{n} \le f_i' \le \frac{1}{m}$ for all $i$.
    \item For all $j = 0, \dots , d$ and for every interval $I$ such that $f_{j} < f_{j+1}$ on $I$, the function

$F_j:=\sum_{0<i \le j} f_i$ is convex and piecewise linear on $I$ with slopes in
\begin{equation}
    Z(j):=\{\frac{L_+}{m}- \frac{L_-}{n}:L_{\pm}\in [0,d_{\pm}]_{\Z}, L_+ + L_- =j \}
\end{equation}
Here $d_+=m$ and $d_-=n$.
\end{enumerate}

(I) is obvious from our construction. So our next goal is to find an appropriate choice $\tau_1$ and $\tau_2$ so that (II) and (III) hold.

The template $\f$ we defined above, restricted on the interval $[b_k-k,b_k]$, has two intervals of linearity, namely 
$$I_1:=[b_k-k,\eta_k(\tau_1,\tau_2)], \text{and } I_2:=[\eta_k(\tau_1,\tau_2),b_k].$$

On the both of them, we have the separation of components of templates as follows:
$$-\infty=:f_0<f_1=f_2=\cdots=f_r<f_{r+1}=\cdots=f_d=f_{d+1}:=+\infty.$$
By the definition of templates, there are two (integer) \textit{intervals of equality}:
$$(0,r]_{\Z}\text{~and~}(r,d]_{\Z},$$
where the subscripts means intersections with the set of integers, forming a partition of the set $\{1,2,\cdots,d\}$.

On the interval of linearity $I_1$, we have slopes:
\begin{equation}
    \begin{cases}
      f_1'(a_k+)=\cdots=f_r'(a_k+)=-\frac{\tau_1}{r}\\
      f_{r+1}'(a_k+)=\cdots=f_r'(a_k+)=+\frac{\tau_1}{d-r}
    \end{cases}\,.
\end{equation}

Now we will follow the procedures of computing the lower and upper contractions rates:
 \renewcommand{\labelenumi}{(\arabic{enumi})}

\begin{enumerate}
    \item On the intervals of linearity $I_1$, we shall determine the possible values for $L_{\pm}:=L_{\pm}(I_1,q)$ as follows:
    \begin{enumerate}
        \item For $q=0$,
        \begin{equation*}
            \begin{cases}
            L_++L_-=0\\
            \frac{L_+}{m}-\frac{L_-}{n}=F'_0=0
         \end{cases}\iff L_+=L_-=0.
        \end{equation*}
        
        \item For $q=d$,
        \begin{equation*}
            \begin{cases}
            L_++L_-=d=m+n\\
            \frac{L_+}{m}-\frac{L_-}{n}=F'_d=0
         \end{cases}\iff L_+=m, L_-=n.
        \end{equation*}
        
        \item For $q=r$,
        \begin{equation*}
            \begin{cases}
            L_++L_-=r\\
            \frac{L_+}{m}-\frac{L_-}{n}=F'_r=-\tau_1         
           
         \end{cases}
          \vspace{5mm}
        \end{equation*}
        Note that in this case, in order to make our solution      
        \begin{align*}
            L_+ &=\frac{rm-mn\tau_1}{m+n}\\
            L_- &=\frac{rn+mn\tau_1}{m+n}
        \end{align*}
        
        satisfy the conditions (II) and (III) in the definition of templates. We need:
        
        \begin{align*}
            & \tau_1>0 \tag{1} \\
            & L_+=\frac{rm-mn\tau_1}{m+n} \in [0,m]_{\Z} \tag{2}\\
            & L_-=r-L_+=\frac{rn+mn\tau_1}{m+n} \in [0,n]_{\Z} \tag{3}\\
            & -\frac{1}{n}\le -\frac{\tau_1}{r} \le \frac{1}{m} \tag{4}\\
            & -\frac{1}{n}\le \frac{\tau_1}{d-r} \le \frac{1}{m} \tag{5}\\  
        \end{align*} 
 Solving (2) for $\tau_1$ and using (1), we get
 \begin{equation}
     \tau_1=\frac{-(m+n)L_+ + rm}{mn}=\frac{r}{n}-L_+(\frac{m+n}{mn})>0
 \end{equation}
 
Now we can list the possible choices for    $L_+, L_-,\tau_1$ in the following table: 

\begin{table}[h]
\begin{center}
\setlength\tabcolsep{15pt} 
\renewcommand{\arraystretch}{1.8}

\begin{tabular}{|c|c|c|c|c|}
\hline
$L_+$  & $0$ & $1$ & $\cdots$ & $\left \lceil \frac{rm}{m+n}-1 \right \rceil$ \\ \hline

$L_-$  & $r$  & $r-1$ & $\cdots$ & $r-\left \lceil \frac{rm}{m+n}-1 \right \rceil $ \\ \hline

$\tau_1$ & $\frac{r}{n}$ & $\frac{r}{n}-\frac{m+n}{mn}$ & $\cdots$ & $(\tau_1)_{\min}$                                     \\ \hline
\end{tabular}
\caption{\label{integer}Correspondence between $L_{\pm}$ and $\tau_1$ on $I_1$.}
\end{center}
\end{table}

Under the condition $\tau_1>0$, the maximal possible value for $L_+$ is $\left \lceil \frac{rm}{m+n}-1 \right \rceil$, which gives us the minimal positive value for $\tau_1$. Note that here we cannot choose $L_+=\left \lfloor \frac{rm}{m+n} \right \rfloor$ since this may result in $\tau_1=0$ if $\frac{rm}{m+n}$ is an integer.

Now we turn to look at the conditions (4) and (5). Given $\tau_1>0$, it is equivalent to saying
$$0<\tau_1\le \min \{ \frac{r}{n}, \frac{d-r}{m}\}$$
        \end{enumerate}
Note that 
$$(\tau_1)_{\min}=  \tau_1=\frac{-(m+n)\left \lceil \frac{rm}{m+n}-1 \right \rceil + rm}{mn}=\frac{r}{n}-\left \lceil \frac{rm}{m+n}-1 \right \rceil(\frac{m+n}{mn})\le \frac{r}{n}$$ 

So it suffices to check if 
$$\frac{r}{n}-\left \lceil \frac{rm}{m+n}-1 \right \rceil(\frac{m+n}{mn})\le \frac{d-r}{m}$$

But this is equivalent to 
\begin{align*}
    &\frac{r}{n}-\frac{d-r}{m} \le \left \lceil \frac{rm}{m+n}-1 \right \rceil \frac{m+n}{mn}\\
    \iff & (\frac{1}{m}+\frac{1}{n})r-\frac{d}{m} \le \left \lceil \frac{rm}{m+n}-1 \right \rceil\\
    \iff & r-n=r-\frac{\frac{d}{m}}{\frac{1}{m}+\frac{1}{n}} \le \left \lceil \frac{rm}{m+n}-1 \right \rceil
\end{align*}
By our lemma above, the last line holds for $r\le \max(m,n)$.

\item On the intervals of linearity $I_2$, we shall determine the possible values for $L_{\pm}:=L_{\pm}(I_2,q)$ as follows:
    \begin{enumerate}
        \item For $q=0$,
        \begin{equation*}
            \begin{cases}
            L_++L_-=0\\
            \frac{L_+}{m}-\frac{L_-}{n}=F'_0=0
         \end{cases}\iff L_+=L_-=0.
        \end{equation*}
        
        \item For $q=d$,
        \begin{equation*}
            \begin{cases}
            L_++L_-=d=m+n\\
            \frac{L_+}{m}-\frac{L_-}{n}=F'_d=0
         \end{cases}\iff L_+=m, L_-=n.
        \end{equation*}
        
        \item For $q=r$,
        \begin{equation*}
            \begin{cases}
            L_++L_-=r\\
            \frac{L_+}{m}-\frac{L_-}{n}=F'_r=+\tau_2           
            \end{cases}\iff            
            \begin{cases}
            L_+ =\frac{rm+mn\tau_2}{m+n}\\
            L_- =\frac{rn-mn\tau_2}{m+n}
            \end{cases}
        \end{equation*} 
    \end{enumerate}
\end{enumerate}

A similar discussion yields the possible values for $L_+,L_-$ and $\tau_2$ as listed in the following table:

\begin{table}[h]
\begin{center}
\setlength\tabcolsep{15pt} 
\renewcommand{\arraystretch}{1.8}

\begin{tabular}{|c|c|c|c|c|}
\hline
$L_-$  & $0$ & $1$ & $\cdots$ & $\left \lceil \frac{rn}{m+n}-1 \right \rceil$ \\ \hline

$L_+$  & $r$  & $r-1$ & $\cdots$ & $r-\left \lceil \frac{rn}{m+n}-1 \right \rceil $ \\ \hline

$\tau_2$ & $\frac{r}{m}$ & $\frac{r}{m}-\frac{m+n}{mn}$ & $\cdots$ & $(\tau_2)_{\min}$                                     \\ \hline
\end{tabular}
\caption{\label{integer}Correspondence between $L_{\pm}$ and $\tau_2$ on $I_2$.}
\end{center}
\end{table}

From the derivative restrictions:
\begin{align*}
        & -\frac{1}{n}\le \frac{\tau_2}{r} \le \frac{1}{m} \\
        & -\frac{1}{n}\le \frac{\tau_2}{d-r} \le \frac{1}{m} 
\end{align*}
we have $\tau \le \min\{\frac{r}{m},\frac{d-r}{n}\}$, but $\tau_2\le \frac{r}{m}$ is automatic and the condition
$$(\tau_2)_{\min}=-\frac{m+n}{mn}\left \lceil \frac{rn}{m+n}-1 \right \rceil+\frac{r}{m}\le \frac{d-r}{n}$$
is equivalent to 
$$r-m \le \left \lceil \frac{rn}{m+n}-1 \right \rceil.$$
Again, this is guaranteed by the proceeding lemma as long as $r\le \max{(m,n)}$. 

Summarizing what we have done so far, we have demonstrated that it is possible to find slopes $\tau_1,\tau_2$ such that the $d$-component piecewise linear function $\f=\q[\tau_1,\tau_2]$ on $[a_k,b_k]=[b_k-k,b_k]$ as constructed is indeed a partial template as long as $r\le \max(m,n)$.

Now we will show that 
$$\de(\f)=mn.$$

Indeed,
$$\de(\f):=\lim_{T\to \infty}\Delta(\f,T):= \lim_{T\to \infty}\frac{1}{T}\int_0^T \de(\f,[0,t])dt,$$
where $0< \de(\f,[a_k,b_k])\le mn$ and $\de(\f,[b_{n-1},a_k])=mn$.

Since $$0=:b_0<b_1<\cdots<b_k=n^n\to \infty,$$
any $T>0$ lies in the dilating period $(b_{n-1},b_k]$ of $\f$ for some $n\ge 0$. In each of such period, the average of integral $\lim_{T\to \infty}\frac{1}{T}\int_0^T \de(\f,[0,T])dt$ reaches its maximum at $T=a_k$ and its minimum at $T=b_k$:
$$\frac{1}{a_k}\int_0^{a_k} \de(\f,[0,t])dt \le \frac{1}{T}\int_0^T \de(\f,[0,t])dt \le \frac{1}{b_k}\int_0^{b_k} \de(\f,[0,t])dt \le mn.$$

But

\begin{align*}
    \frac{1}{a_k}\int_0^{a_k}\de(\f,[0,t])dt
    \ge & \frac{1}{a_k} \sum_{i=1}^n \int_{b_{i-1}}^{a_i} \de(\f,[0,t])dt\\
    = & \frac{1}{a_k} \sum_{i=1}^n \int_{b_{i-1}}^{a_i} mn~dt \\
    = & mn \frac{1}{k^k-k} \sum_{i=1}^k (i^i-i-(i-1)^{i-1})\\
    = & mn \frac{k^k}{k^k-k} - mn \frac{k(k+1)}{2(k^k-k)}
\end{align*}
which converges to zero as $k\to \infty$.

Therefore, by the variational principle, 
$$\dim_P(\BA(r+1)-\BA(r))=\dim_H(\BA(r+1)-\BA(r))=mn,$$
for all $r\le \max(m,n)$. Since $\frac{d}{2} \le \max(m,n)$, we have proved the theorem for at least half of the orders.

For the computation of cases when $r>\max(m,n)$, we will study the dual of the orbit of lattices $(g_tu_A\Z^d)_{t\ge 0}$.

Recall that for a lattice $\Lambda\subset \R^d$ with basis $\{\bb_1,\cdots,\bb_d\}$, since $\bb_1,\cdots,\bb_d$ are linearly independent, from linear algebra we know there exist vectors $\bb_1^*, \cdots \bb_d^*$, call \textit{dual vectors} to 
$\bb_1,\cdots,\bb_d$, such that
$$\langle \bb_i,\bb_i^*\rangle=
\begin{cases}
0 & i\ne j\\
1 & i=j
\end{cases}.
$$

\vspace{5mm}
The $\Z$-span of dual basis vectors, namely $\Lambda^*:=\text{Span}\{\bb_1^*, \cdots \bb_d^*\}$, is called the \textit{dual (or polar or reciprocal) lattice} to the lattice $\Lambda$. \label{dual lattice}

Although defined through basis, it turns out that the dual lattices are independent of the choice of basis of the original lattice.

\begin{proposition}
The dual lattice $\Lambda^*$ consists of all vectors $\bb^* \in \R^d$ such that $\langle \bb^*,\bb \rangle$ is an integer for all $\bb$ in $\Lambda$. As a consequence, $\Lambda^*$ is also the dual of $\Lambda$.
\end{proposition}

\begin{proof}
Let $\bb_1,\cdots,\bb_d$ be a basis of the lattice $\Lambda$ and their duals be $\bb_1^*,\cdots,\bb_d^*$. For any $\bb \in \Lambda$ and any $\cc \in \Lambda^*$, suppose
$$\bb=s_1\bb_1+\cdots+s_d\bb_d, \text{ and } \cc=t_1\bb_1^*+\cdots+t_d\bb_d^*$$

with integer coefficients $s_i,t_i\in \Z$ for $i=1,2,\cdots, d.$ We have immediately that 
$$\langle \bb,\cc \rangle = s_1t_1\cdots+s_d t_d\in \Z.$$

On the other hand, if $\bb^*= u_1\bb_1^*+\cdots+u_d\bb_d^* \in \R^d$, where $u_i \in \R$ satisfies $\langle \bb^*,\bb \rangle \in \Z,$ for any $\bb \in \Z$, then in particular this holds for  $\bb=\bb_i$, for any $i=1,2,\cdots,d$ and thus
$$u_i=\langle \bb^*,\bb_i \rangle  \in \Z. $$
Therefore $\bb^* \in \Lambda^*$.

\end{proof}

The dual lattice operator commutes nicely with an invertible linear transformation on $\R^d$:

\begin{proposition}\label{dual is same as transpose inverse}
Let $\Lambda$ be a lattice on $\R^d$ and $T:\R^d \to \R^d$ be an invertible linear transformation, then we have
$$(T\Lambda)^*=T^*\Lambda^*,$$
where $T^*={}^tT^{-1}$ is the inverse of the transpose of $T$ and $\Lambda^*$ is the dual lattice to $\Lambda$. 
\end{proposition}

\begin{proof}
If $\bb_1,\cdots, \bb_d$ is a basis of $\Lambda$, then $\bb_1,\cdots,T\bb_d$
is a basis of $T\Lambda$. The corresponding dual basis
$$(T\bb_1)^*,\cdots,(T\bb_d)^*$$
satisfy 
\begin{equation*}
\begin{bmatrix}
{}^t(T\bb_1)^* \\
\vdots \\
{}^t(T\bb_1)^* 
\end{bmatrix}  
\begin{bmatrix}
T\bb_1 & \cdots & T\bb_d 
\end{bmatrix} =I_d 
\end{equation*}
But on the other hand,
\begin{equation*}
     \begin{bmatrix}
{}^t({}^tT^{-1}\bb_1) \\
\vdots \\
{}^t({}^t T^{-1}\bb_d) 
\end{bmatrix}  
\begin{bmatrix}
T\bb_1 & \cdots & T\bb_d 
\end{bmatrix} =I_d 
\end{equation*}
So by the uniqueness of inverse matrix, $T^*\bb_i={}^t T^{-1}\bb_i= (Tb_i)$, for any $i=1,2\cdots,d$. Therefore $(T\Lambda)^*=T^*\Lambda^*.$
\end{proof}

We are able to use the idea of dual lattice to address the issue of higher $r$'s, thanks to the following theorem:

\begin{theorem}[\cite{CA97} Chapter VIII, Theorem VI]\label{sucessive minima of dual lattice}
Let $\lambda_1,\cdots,\lambda_d$ be the successive minima of lattices in $\R^d$. Then for a lattice $\Lambda$ and its dual $\Lambda^*$, we have
$$1\le \lambda_r(\Lambda) \lambda_{d+1-r} (\Lambda^*) \le d!$$
for any $r=1,2,\cdots d$.
\end{theorem}

Now let us return to our proof, for the flow of lattices $(g_tu_A\Lambda)$, the proposition above gives its dual as:
\begin{align*}
    (g_tu_A\Z^d)^*= & g_t^* u_A^*(\Z^d)^* \\
                     = & {}^t g_t^{-1} \cdot {}^t u_A^{-1} \Z^d\\
                     = & {}^t\begin{bmatrix}e^{t/m}I_m & 0 \\0 & e^{-t/n}I_n \end{bmatrix}^{-1}\cdot
                     {}^t\begin{bmatrix}I_m & A \\0 & I_n \end{bmatrix}^{-1}\Z^d\\
                     = & \begin{bmatrix}e^{-t/m}I_m & 0 \\0 & e^{t/n}I_n \end{bmatrix}\cdot
                     \begin{bmatrix}I_m & 0 \\-{}^tA & I_n \end{bmatrix}\Z^d\\
\end{align*}

Now observe that for $r> \frac{d}{2}$, $d+1-r < d+1-\frac{d}{2}=\frac{d}{2}+1\le \max(m,n)+1$ (which is the same as $d\le \max(m,n)$), and we have that
\begin{align*}
         & \liminf_{t\to \infty} \h_{A,r}=-\infty \\
    \iff & \liminf_{t \to \infty} \lambda_{r}(g_tu_A\Z^d)=0 \\
    \iff & 
    \liminf_{t \to \infty} \lambda_r(g_tu_A\Z^d)=0\\
    \iff &
    \limsup_{t \to \infty} \lambda_{d+1-r}((g_tu_A\Z^d)^*)= \infty\\
    \iff &
    \limsup_{t \to \infty} \lambda_{d+1-r}\left(\begin{bmatrix}e^{-t/m}I_m & 0 \\0 & e^{t/n}I_n \end{bmatrix}\cdot\begin{bmatrix}I_m & 0 \\-{}^tA & I_n \end{bmatrix}\Z^d \right ) =\infty \\
    \iff & 
    \limsup_{t \to \infty} \h_{-^tA,d+1-r}(t)=\infty \\
    \iff &
    \liminf_{t \to \infty} -\h_{-^tA,d+1-r}(t)=-\infty
\end{align*}
and that
\begin{align*}
    &\liminf_{t \to \infty} \h_{A,r+1}(t)>-\infty \\
    \iff &
    \liminf_{t \to \infty} \lambda_{r+1}(g_tu_A\Z^d)>0 \\
    \iff & 
    \liminf_{t \to \infty} \lambda_{r+1}(g_tu_A\Z^d)>0\\
    \iff &
    \limsup_{t \to \infty} \lambda_{d-r}((g_tu_A\Z^d)^*)< \infty\\
    \iff &
    \limsup_{t \to \infty} \lambda_{d-r}\left(\begin{bmatrix}e^{-t/m}I_m & 0 \\0 & e^{t/n}I_n \end{bmatrix}\cdot\begin{bmatrix}I_m & 0 \\-{}^tA & I_n \end{bmatrix}\Z^d \right ) <\infty \\
    \iff & 
    \limsup_{t \to \infty} \h_{-^tA,d-r}(t)<\infty \\
    \iff &
    \liminf_{t \to \infty} -\h_{-^tA,d-r}(t)>-\infty
\end{align*}

Therefore, by replacing templates $\f$ with $-\f$ in our $\mathcal{S}$ and observe that if we change any $m\times n$ matrix $A$ $D(\mathcal{S})$ to its negative transpose $-{t}^A$, then the Hausdorff and Packing dimensions of $D(\mathcal{S})$ will not change and the result we obtained above for $d\le \max(m,n)$ applies. This proves 
$$\dim_P(\BA(r+1)-\BA(r))=\dim_H(\BA(r+1)-\BA(r))=mn,$$
for $d>\max(m,n)$ and completes the proof of the theorem.

\end{proof}

\section*{Acknowledgement}
The author would like to thanks Professors Tushar Das, Dmitry Kleinbock, Nimish Shah and Barak Weiss for the encouragement and helpful discussions on this project. 

\printbibliography[
heading=bibintoc,
title={Bibliography}
]
\end{document}